\documentclass[10pt,twocolumn,twoside]{IEEEtran}
\ifCLASSINFOpdf
\else
\fi
\usepackage{amsmath,mathrsfs}
\usepackage{amssymb,amsthm}
\usepackage{graphicx}
\usepackage{paralist}
\usepackage{caption}
\usepackage{pstricks,pst-all}
\usepackage{mathdots}
\usepackage{subfigure}

\DeclareMathOperator{\aut}{Aut}

\DeclareMathOperator{\diag}{diag}
\DeclareMathOperator{\spn}{span}
\DeclareMathOperator{\img}{range}
\DeclareMathOperator{\rank}{rank}

\newcommand{\real}{\mathbb{R}}

\newcommand{\lap}{L}
\newcommand{\wlap}{\widetilde{\lap}}

\newcommand{\vv}{x}
\newcommand{\ev}{e}
\newcommand{\bv}{b}
\newcommand{\ez}{y}
\newcommand{\ones}{e}
\newcommand{\ptw}{\pi^*}

\newcommand{\zeros}{0}
\newcommand{\wgph}{\widetilde{G}}
\newcommand{\dotprod}[1]{\left\langle #1 \right\rangle}

\newtheorem{theorem}{Theorem}[section]

\newtheorem{proposition}{Proposition}[section]
\newtheorem{lemma}{Lemma}[section]
\theoremstyle{definition}
\newtheorem{definition}{Definition}[section]
\newtheorem{example}{Example}[section]
\newtheorem{remark}{Remark}[section]

\begin{document}

\title{Strongly uncontrollable network topologies}
\author{Cesar~O.~Aguilar%
\thanks{C.O.\ Aguilar is with the Department of Mathematics, State University of New York, Geneseo, NY (aguilar@geneseo.edu)}}

\maketitle


\begin{abstract}
In this paper, we present a class of network topologies under which the Laplacian consensus dynamics exhibits undesirable controllability properties under a broadcast control signal.  Specifically, the networks we characterize are uncontrollable for any subset of the nodes chosen as control inputs and that emit a common control signal.  We provide a sufficient condition for a network to contain this strong uncontrollability property and describe network perturbations that leave the uncontrollability property invariant.  As a by-product, we identify non-trivial network topologies that require the control of approximately half the nodes in the network as a necessary condition for controllability.
\end{abstract}


%

\section{Introduction}
%
%
%
%

\IEEEPARstart{C}{ontrollability} of networked multi-agent systems is an ongoing topic of research in the control systems community due to the proliferation of technologies associated with large-scale network models, see for instance \cite{Tanner:04-cdc, AR-MJ-MM-ME:09, SM-ME-AB:10, GP-GN:12, AC-MN-MM:14,SZ-MC-MKC:14,NM-SZ-KC:14,AC-MM:14,CA-BG:14a,CA-BG:17,AY-WA-ME:16,SH:17,YG-LW:18} and references therein.  A primary goal of the ongoing research has been to identify graph-theoretic structures that are possible obstructions to controlling a networked control system.  To that end, the lack of controllability has been primarily attributed to the existence of so-called \textit{equitable partitions} of the vertex set \cite{AR-MJ-MM-ME:09,SM-ME-AB:10, AC-MM:14, SZ-MC-MKC:14,CA-BG:17}.  We note that a special case of these equitable partitions is the presence of structural symmetries in the associated network model \cite{AR-MJ-MM-ME:09}.  Roughly speaking, the existence of an equitable partition of the nodes of a network induce an invariant subspace for the uncontrolled dynamics and thus if the control nodes are chosen to preserve the invariant subspace then uncontrollability ensues.  A closely related line of research is the so-called \textit{minimal controllability problem} which is concerned with the scenario of controlling a large-scale multi-agent system with the fewest possible number of control nodes \cite{AO:14, AO:15, SP-SK-PA:15, SP-SK-PA:16,SP-GR-SK-PA-JR:17}.  Although it has been shown that solving minimal controllability problems is computationally intractable for generic systems (unless $P=NP$), heuristic algorithms are known that produce approximate solutions \cite{AO:14, SP-SK-PA:15, SP-SK-PA:16,SP-GR-SK-PA-JR:17}.  On the other hand, in the case of structured systems, the minimal controllability problem can be solved in polynomial time \cite{CC-JD:13,AO:15}.  

In this paper, we are motivated by the following question:  What structural properties present in a network result in uncontrollable dynamics for \textit{any} choice of control nodes emitting a common control signal?  Specifically, we focus on networked multi-agent systems undergoing the Laplacian consensus dynamics and describe network topologies that are uncontrollable under a  broadcast control for any choice of leader nodes.  Such networks were introduced in \cite{CA-BG:14a} and were called \textit{strongly uncontrollable graphs}.  To be more precise about the issue at hand, let $G$ be a simple $n$-vertex graph with Laplacian matrix $L=L(G)=A(G)-D(G)$, where $A(G)$ is the adjacency matrix and $D(G)$ is the diagonal degree matrix of $G$, and let $b$ be a binary vector.  A trivial necessary condition for the pair $(L, b)$ to be controllable is that the eigenvalues of $L$ are all distinct \cite[pg. 95]{ES:98}.  The authors in \cite{CA-BG:14a}, however, provide examples of graphs for which $L$ has \textit{distinct} eigenvalues but $(L,b)$ is uncontrollable for \textit{every} choice of binary vector $b$ and such graphs were called strongly uncontrollable graphs\footnote{In \cite{CA-BG:14a} (Theorem 4.1), there is also a construction of a class of graphs such that $(L,b)$ is uncontrollable for every choice of binary vector $b$ but such graphs contain \textit{repeated} eigenvalues.}.  A similar definition of \textit{strong uncontrollability} can be made using the adjacency matrix $A(G)$, the signless Laplacian matrix $Q(G)=A(G)+D(G)$, or some other graph matrix relevant to the network model in consideration.  In the case of the Laplacian matrix, we have performed an exhaustive numerical search revealing that strongly uncontrollable graphs do not appear until the number of nodes is $n=8$, that is, no connected graph is strongly uncontrollable for $n\leq 7$ in the case of $L$.  Our numerical search produced an enumeration of strongly uncontrollable graphs for $8\leq n\leq 12$ and the results are shown in Table~\ref{tab:strongly uncontrollable-graphs}.  Interestingly, for the case of the adjacency matrix, our numerical search revealed that no connected graph for $n\leq 10$ is strongly uncontrollable and, based on our results in this paper, we conjecture that strongly uncontrollable graphs using $A(G)$ do not exist.  

\begin{table}[h!]
\centering
\begin{tabular}{cccccc}\hline
$n$ & 8&9&10&11&12 \\ \hline
$s(n)$ &10&12&91&232&1749 \\ \hline
\end{tabular}
\caption{$s(n)$ is the number of connected graphs on $n$ vertices such that $L=L(G)$ has distinct eigenvalues and $(L,b)$ is uncontrollable for \textit{every} $b\in \{0,1\}^n$}\label{tab:strongly uncontrollable-graphs}
\end{table}


The main contribution of this paper is the characterization of a class of network topologies that result in a strongly uncontrollable multi-agent system undergoing the Laplacian consensus dynamics.  The identification of these topologies could provide a test bed to narrow the gap between known sufficient conditions and necessary conditions for network controllability.  As a by-product of our results, we identify a class of non-trivial network topologies that require the control of approximately half of the nodes for any chance of controllability.  The topologies we study are of interest since they contain many ``local symmetries'' that are similar to the symmetries found in large-scale real-world complex networks \cite{BM-RS-JA:08}.  The discovered network topologies contain two main structural ingredients, namely, the presence of many so-called \textit{twin} nodes and certain equitable partitions.  Twin vertices induce so-called Faria eigenvectors \cite{TB-JL-PS:07} of the Laplacian matrix $L$ and, through the use of the classical Popov-Belevitch-Hautus (PBH) test for controllability, result in the uncontrollability of $(L,b)$ for many choices of $b$ but in general are not enough for strong uncontrollability.  The second main ingredient is the presence of certain equitable partitions \cite{DC-CD-PR:07} which, together with the presence of twin nodes, result in a strongly uncontrollable graph.  Graph vertex partitions are becoming a standard tool used to study control-theoretic properties in multi-agent systems, see for instance \cite{AR-MJ-MM-ME:09,SM-ME-AB:10,SZ-MC-MKC:14,NM-HL-KC:14,NM-SZ-KC:15} and references therein.  Graph vertex partitions also take an important role in the study of synchrony and pattern formation in coupled cell networks \cite{IS-MG-MP:03,MG-IS-AT:05}.  In addition to our main result, and motivated by the recent research activity with structural controllability, we identify graph perturbations which preserve the strong uncontrollability property.

This paper is organized as follows.  In Section~\ref{sec:prelim} we introduce our notation, review the connection between equitable partitions of a graph and controllability, and present some technical results.  In Section~\ref{sec:twins}, we introduce twin graphs which are non-trivial network topologies requiring that approximately half of the nodes be controlled.  In Section~\ref{sec:main} we present our main result of the paper, namely, a sufficient condition for a networked multi-agent system undergoing the Laplacian consensus dynamics to be uncontrollable for any choice of leader nodes under a broadcast control signal.  In Section~\ref{sec:perturbations}, we identify three vertex addition operations that leave the strong uncontrollability property invariant.  We end the paper with a Conclusion.

\section{Preliminaries}\label{sec:prelim}
In this section, we introduce our notation, review the notion of equitable partitions and their role in network controllability, and present some technical results.

\subsection{Notation}
The all ones vector in $\real^n$ is denoted by $e=(1,1,\ldots,1)$ and the context will make it clear the value of $n$.  For $u\in \{1,2,\ldots,n\}$ we denote by $e_u$ the unit standard basis vector in $\real^n$ with non-zero entry at $u$.  We say that a square matrix $M$ has \textit{simple spectrum} if every eigenvalue of $M$ has algebraic multiplicity one.  The column/range space of $M$ will be denoted by $\img(M)$.  We equip $\real^n$ with the standard Euclidean inner product $\langle x,y\rangle=x^Ty$ and denote by $\Omega^\perp=\{x\in\real^n\;|\; \langle x,y\rangle=0,\; \forall\; y\in\Omega\}$ the orthogonal complement of a set $\Omega\subseteq\real^n$.  A subspace $W\subset\real^n$ is said to be \textit{$M$-invariant} if $w\in W$ implies that $Mw \in W$.  For each positive integer $n$, we let $\{0,1\}^n$ denote the set of binary vectors of length $n$.  Finally, we let $\mathbb{N}_0=\{0,1,2,\ldots\}$.

By a \textit{weighted digraph} we mean a triple $G=(V,E,\phi)$ where $V$ is the set of vertices, $E\subset V\times V$ is the set of directed edges (or arcs), and $\phi:E\rightarrow\mathbb{N}_0$ is the weight function on the arcs with the property that $\phi(u,v)\neq 0$ if and only if $(u,v)\in E$.   We do not have a need for loops in a graph and thus $(u,u)\notin E$ for all $u\in V$.  If $(u,v)\in E$ if and only if $(v,u) \in E$ and $\phi\equiv 1$ then we call $G$ a \textit{simple} graph and instead use the usual notation $G=(V,E)$.  The context will make it clear as to whether $G$ is a simple graph or a weighted digraph.

Let $G=(V,E,\phi)$ be a weighted digraph.  The \textit{neighborhood} of $u\in V$ is the set $N(u)=\{v\in V\;|\; (u,v) \in E\}$ and the \textit{degree} of $u$ is $\deg(u) = \sum_{v\in V} \phi(u,v)$.  For a simple graph $G$, $\deg(u)=|N(u)|$ for all $u\in V$, i.e., the number of vertices in $N(u)$.  More generally, for a subset $C\subseteq V$ we define the degree of $u$ in $C$ by
\[
\deg(u, C):=\sum_{v\in C} \phi(u, v).
\]
If $G$ is a simple graph then $\deg(u,C)=|N(u)\cap C|$.  If we label the vertices as $V=\{v_1,v_2,\ldots,v_n\}$, the adjacency matrix $A=A(G)$ of $G$ has entries $A_{i,j} = \phi(v_i, v_j)$.  The degree matrix of $G$ is the diagonal matrix $D=D(G)$ whose $i$th diagonal element is $\deg(v_i)$, and the Laplacian matrix of $G$ is $L(G) = D - A$.  

Finally, we recall that a linear single-input time-invariant system $\dot{x}=Ax+bu$ on $\real^n$ is controllable if and only if the smallest $A$-invariant subspace containing $b$ is all of $\real^n$, that is, $\spn\{b,Ab,\ldots,A^{n-1}b\}=\real^n$.  A well-known characterization of controllability is the Popov-Belevitch-Hautus (PBH) eigenvector test which states that $(A,b)$ is controllable if and only if $\xi^T b \neq 0$ for every eigenvector $\xi$ of $A^T$.

%

\subsection{Almost equitable partitions}


Let $G$ be a weighted digraph with vertex set $V=\{v_1,v_2,\ldots,v_n\}$.  The \textit{characteristic vector} of a subset $C\subset V$ is the binary vector $\xi\in\{0,1\}^n$ such that $\xi_i = 1$ if and only if $v_i \in C$.  Let $\pi=\{C_1,C_2,\ldots,C_k\}$ be a set partition of $V$, that is, $C_i \cap C_j=\emptyset$ if $i\neq j$ and $\bigcup_{i=1}^k C_i = V$.   The subsets $C_i$ will be called \textit{cells} of $\pi$.  The \textit{characteristic matrix} of $\pi$ is the $n\times k$ matrix $P_\pi$ whose $i$th column is the characteristic vector of $C_i$, for $i=1,\ldots,k$.  When no confusion arises we denote $P_\pi$ simply by $P$.  We say that $\pi$ is an \textit{almost equitable partition} (AEP) of $G$ if for every distinct ordered pair of cells $(C_i,C_j)$ it holds that $\deg(u, C_j) = \deg(v, C_j)$ for every $u,v \in C_i$.  In this case, we define $\deg(C_i, C_j):= \deg(u, C_j)$ for some (and hence all) $u\in C_i$.  We note that in general, $\deg(C_i, C_j) \neq \deg(C_j, C_i)$.  If $\pi$ is an AEP of $G$, the \textit{quotient graph $G_\pi$ of $G$ with respect to $\pi$} is the weighted digraph with vertex set $V(G_\pi)=\pi$ and arcs $(C_i, C_j) \in E(G_\pi)$ if and only if $\deg(C_i, C_j) \neq 0$ with arc weight $\phi(C_i,C_j) = \deg(C_i, C_j)$.  We denote the adjacency matrix of $G_\pi$ by $A_\pi = A(G_\pi)$ and its Laplacian matrix by $L_\pi=L(G_\pi)$, and we note that $A_\pi$ and $L_\pi$ are generally non-symmetric matrices.  An almost equitable partition $\pi=\{C_1,\ldots,C_k\}$ of $G$ is called an \textit{equitable partition} if in addition $\deg(u,C_i) = \deg(v, C_i)$ for all $u,v\in C_i$ and all $i=1,\ldots,k$.  Hence, the extra requirement imposed on an equitable partition is that vertices in any given cell $C_i$ have equal degree \textit{within} $C_i$.  To make this paper as self-contained as possible, below we provide an example of the previous notions.
\begin{example}
Consider the simple graph $G$ shown in Figure~\ref{fig:aep-example} with $n=11$ vertices.  Consider the vertex partition $\pi=\{C_1,C_2,C_3,C_4\}$ where $C_1=\{v_1,v_2,v_3,v_4\}$, $C_2=\{v_5,v_6\}$, $C_3=\{v_7,v_8,v_9,v_{10}\}$, and $C_4=\{v_{11}\}$.  The edges within each cell are displayed as dashed lines.   The reader is invited to verify that for any pair of distinct cells $(C_i, C_j)$ it holds that $\deg(v, C_j) = \deg(u,C_j)$ for every $u, v \in C_i$.   Notice that, for instance, $\deg(v_1, C_1) = 1$ while $\deg(v_3, C_1)=2$, and thus $\pi$ is an AEP of $G$ but not an equitable partition. The characteristic matrix $P$ of $\pi$ is the $n\times k = 11\times 4$ matrix
\[
P^T = 
\begin{bmatrix}
1&1&1&1&0&0&0&0&0&0&0\\
0&0&0&0&1&1&0&0&0&0&0\\
0&0&0&0&0&0&1&1&1&1&0\\
0&0&0&0&0&0&0&0&0&0&1
\end{bmatrix}.
\]
The quotient graph $G_\pi$ is also shown in Figure~\ref{fig:aep-example} and the adjacency and Laplacian matrix of $G_\pi$ are
\[
A_\pi = \begin{bmatrix}
0 & 1 & 0 & 0\\
2 & 0 & 2 & 1\\
0 & 1 & 0 & 1\\
0 & 2 & 4 & 0
\end{bmatrix},\quad L_\pi = \begin{bmatrix}
1 & -1 & 0 & 0\\
-2 & 5 & -2 & -1\\
0 & -1 & 2 & -1\\
0 & -2 & -4 & 6
\end{bmatrix}.
\]
We note that $A_\pi$ and $L_\pi$ are non-symmetric matrices. $\square$
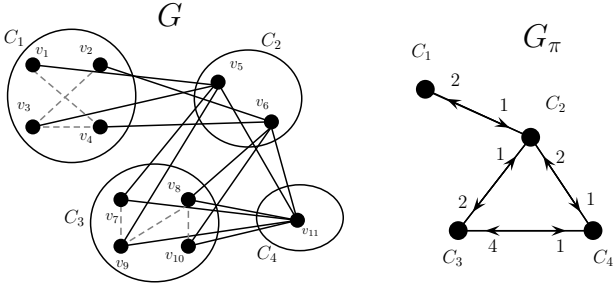
\begin{figure}[t]
\centering
\psscalebox{0.5 0.5} 
{
\begin{pspicture}(0,-3.5372415)(16.062069,3.5372415)
\definecolor{colour0}{rgb}{0.5,0.5,0.5}
\psline[linecolor=black, linewidth=0.04](7.8,-1.8820689)(4.903448,-2.5717242)
\psline[linecolor=black, linewidth=0.04](7.8,-1.8820689)(4.903448,-1.3303448)
\psline[linecolor=black, linewidth=0.04](7.8,-1.8820689)(3.110345,-1.3303448)
\psline[linecolor=black, linewidth=0.04](7.8,-1.8820689)(3.110345,-2.5717242)
\psline[linecolor=black, linewidth=0.04](7.110345,0.6006897)(7.8,-1.8820689)
\psline[linecolor=black, linewidth=0.04](5.7310343,1.704138)(7.8,-1.8820689)
\psline[linecolor=black, linewidth=0.04](7.110345,0.6006897)(4.903448,-1.3303448)
\psline[linecolor=black, linewidth=0.04](7.110345,0.6006897)(4.903448,-2.5717242)
\psline[linecolor=black, linewidth=0.04](5.5931034,1.704138)(3.110345,-1.3303448)
\psline[linecolor=black, linewidth=0.04](5.7310343,1.704138)(3.110345,-2.5717242)
\psline[linecolor=black, linewidth=0.04](5.7310343,1.704138)(0.76551723,2.2558622)
\psline[linecolor=black, linewidth=0.04](5.5931034,1.704138)(0.76551723,0.6006897)
\psline[linecolor=black, linewidth=0.04](7.110345,0.73862076)(2.5586207,0.6006897)
\psline[linecolor=black, linewidth=0.04](7.110345,0.73862076)(2.5586207,2.2558622)
\psellipse[linecolor=black, linewidth=0.04, dimen=outer](1.8,1.428276)(1.7241379,1.7931035)
\psellipse[linecolor=black, linewidth=0.04, dimen=outer](6.489655,1.3593105)(1.4482758,1.3103448)
\psellipse[linecolor=black, linewidth=0.04, dimen=outer](4.0068965,-1.9510344)(1.7241379,1.5862069)
\psellipse[linecolor=black, linewidth=0.04, dimen=outer](7.8689656,-1.8131033)(1.1724138,0.8965517)
\rput[bl](0.8413793,2.4696553){\large $v_1$}
\rput[bl](2.0068965,2.5317242){\large $v_2$}
\rput[bl](0.35172415,1.0144829){\large $v_3$}
\rput[bl](2.0068965,0.18689662){\large $v_4$}
\rput[bl](6.0068965,2.1179311){\large $v_5$}
\rput[bl](6.696552,1.0903449){\large $v_6$}
\rput[bl](2.6965518,-1.9441378){\large $v_7$}
\rput[bl](4.351724,-1.1165516){\large $v_8$}
\rput[bl](2.9724138,-3.185517){\large $v_9$}
\rput[bl](4.289655,-2.9855173){\large $v_{10}$}
\rput[bl](7.924138,-2.295862){$\large v_{11}$}
\rput[bl](0.0,2.7593105){\Large$C_1$}
\rput[bl](6.8344827,2.6834483){\Large$C_2$}
\rput[bl](1.6,-2.02){\Large$C_3$}
\rput[bl](6.724138,-2.9855173){\Large$C_4$}
\psline[linecolor=colour0, linewidth=0.04,linestyle=dashed](3.110345,-1.4682758)(3.110345,-2.5717242)
\psline[linecolor=colour0, linewidth=0.04,linestyle=dashed](3.110345,-2.5717242)(4.903448,-1.4682758)
\psline[linecolor=colour0, linewidth=0.04,linestyle=dashed](4.903448,-1.3303448)(4.903448,-2.709655)
\psline[linecolor=colour0, linewidth=0.04,linestyle=dashed](0.76551723,2.1179311)(2.5586207,0.6006897)
\psline[linecolor=colour0, linewidth=0.04,linestyle=dashed](2.5586207,2.2558622)(0.76551723,0.6006897)
\psline[linecolor=colour0, linewidth=0.04,linestyle=dashed](0.76551723,0.6006897)(2.5586207,0.6006897)
\psdots[linecolor=black, dotsize=0.4](0.76551723,2.2558622)
\psdots[linecolor=black, dotsize=0.4](2.5586207,2.2558622)
\psdots[linecolor=black, dotsize=0.4](0.76551723,0.6006897)
\psdots[linecolor=black, dotsize=0.4](2.5586207,0.6006897)
\psdots[linecolor=black, dotsize=0.4](3.110345,-1.3303448)
\psdots[linecolor=black, dotsize=0.4](3.110345,-2.5717242)
\psdots[linecolor=black, dotsize=0.4](4.903448,-2.5717242)
\psdots[linecolor=black, dotsize=0.4](4.903448,-1.3303448)
\psdots[linecolor=black, dotsize=0.4](5.696552,1.7937932)
\psdots[linecolor=black, dotsize=0.4](7.110345,0.73862076)
\psdots[linecolor=black, dotsize=0.4](7.8,-1.8820689)
\psdots[linecolor=black, dotsize=0.5](11.2,1.6075863)    
\psdots[linecolor=black, dotsize=0.5](12.075862,-2.157931)  
\psdots[linecolor=black, dotsize=0.5](15.662069,-2.157931)  
\psdots[linecolor=black, dotsize=0.5](14.006897,0.32482764)  
\psline[linewidth=1pt,arrowsize=5pt 3]{->}(11.2,1.6075863)(13.4455176,0.58137937)\rput[c](13.3,1.2){\Large $1$} 
\psline[linewidth=1pt,arrowsize=5pt 3]{->}(14.006897,0.32482764)(11.7613794,1.35103457)\rput[c](12,1.85){\Large $2$} 
\psline[linewidth=1pt,arrowsize=5pt 3]{->}(14.006897,0.32482764)(12.462069,-1.66137927)\rput[c](12.2,-1.4){\Large $2$} 
\psline[linewidth=1pt,arrowsize=5pt 3]{->}(12.075862,-2.157931)(13.62069,-0.17172409)\rput[c](13.2,-0.1){\Large $1$} 
\psline[linewidth=1pt,arrowsize=5pt 3]{->}(12.075862,-2.157931)(14.9448276,-2.157931)\rput[c](14.8,-2.55){\Large $1$} 
\psline[linewidth=1pt,arrowsize=5pt 3]{->}(15.662069,-2.157931)(12.7931034,-2.157931)\rput[c](13,-2.55){\Large $4$} 
\psline[linewidth=1pt,arrowsize=5pt 3]{->}(14.006897,0.32482764)(15.3310346,-1.66137927)\rput[c](15.6,-1.3){\Large $1$} 
\psline[linewidth=1pt,arrowsize=5pt 3]{->}(15.662069,-2.157931)(14.3379314,-0.17172409)\rput[c](14.8,-0.2){\Large $2$} 
\psline[linecolor=black, linewidth=0.04](12.075862,-2.157931)(14.006897,0.32482764)(15.662069,-2.157931)(12.075862,-2.157931)
\psline[linecolor=black, linewidth=0.04](11.248276,1.5662069)(14.006897,0.32482764)
\rput[bl](10.834483,2.393793){\Large $C_1$}
\rput[bl](14.4,1){\Large$C_2$}
\rput[bl](11.662069,-3.1){\Large$C_3$}
\rput[bl](15.662069,-3.1){\Large$C_4$}
\rput[bl](13.793103,2.6075864){\Huge $G_{\pi}$}
\rput[bl](4.075862,3.2972414){\Huge $G$}
\end{pspicture}
}
\caption{An example of a graph $G$ with an AEP $\pi$ and the induced quotient graph $G_\pi$}\label{fig:aep-example}
\end{figure}
\end{example}

The relationship between AEPs and invariant subspaces of the Laplacian matrix $L$ is the following.
\begin{theorem}[\cite{DC-CD-PR:07}]\label{thm:equi-lap}
Let $G$ be a weighted digraph and let $\pi=\{C_1,C_2,\ldots,C_k\}$ be a partition of $V(G)$ with characteristic matrix $P$.  Then $\pi$ is an AEP of $G$ if and only if $\img(P)$ is $L$-invariant.  In this case, the Laplacian matrix of the quotient graph $G_\pi$ is $\lap_{\pi}=(P^TP)^{-1}P^T\lap P$.
\end{theorem}

\begin{remark}\label{rmk:eig-struct}
We make the following important observations as a consequence of Theorem~\ref{thm:equi-lap} and which explains the connection between AEPs and obstructions to controllability.  Let $G=(V,E)$ be a simple $n$-vertex graph, let $\pi=\{C_1,\ldots,C_k\}$ be an AEP of $G$, and let $P$ be the characteristic matrix of $\pi$.  Since the subspace $\img(P)\subset\real^n$ is $L$-invariant, there exists a basis for $\img(P)$ consisting of eigenvectors of $L$.  Moreover, since $L$ is symmetric then $\img(P)^{\perp}=\ker(P^T)$ is also $L$-invariant and thus $\ker(P^T)\subset\real^n$ also has a basis of eigenvectors of $L$.  Now, by definition of $P$, $x\in\img(P)$ if and only if the components of $x$ are constant on each cell of $\pi$, that is, $\forall\; v_i,v_j\in C_\ell $ we have $x_i=x_j$ and this holds for all $\ell=1,\ldots,k$.  Similarly, $x\in\ker(P^T)$ if and only if the components of $x$ sum to zero on the cells of $\pi$, that is, $\sum_{v_j \in C_\ell} x_j = 0$ for each $\ell=1,\ldots,k$.  Hence, if $\pi$ is an AEP of $G$ and $L$ is a symmetric matrix, then the eigenvectors of $L$ can be divided into two classes: one class is contained in $\img(P)$ and are characterized by having a constant value on each cell of $\pi$, and the second class is contained in $\ker(P^T)$ and are characterized by summing to zero on each cell of $\pi$.  Hence, if $b\in\{0,1\}^n$ is constant on the cells of $\pi$, that is $b=P\bar{b}$ for some $\bar{b}\in\{0,1\}^k$, then $b$ is clearly orthogonal to the eigenvectors of $L$ contained in $\ker(P^T)$ and therefore by the PBH test $(L,b)$ is uncontrollable.  We illustrate our remarks with our running example.
\end{remark}

\begin{example}
Consider again the simple graph $G$ shown in Figure~\ref{fig:aep-example}, with AEP $\pi=\{C_1,C_2,C_3,C_4\}$ where $C_1=\{v_1,v_2,v_3,v_4\}$, $C_2=\{v_5,v_6\}$, $C_3=\{v_7,v_8,v_9,v_{10}\}$, and $C_4=\{v_{11}\}$.  Hence, $x\in \img(P_\pi)$ if and only if $x = (\alpha,\alpha,\alpha,\alpha,\beta,\beta,\gamma,\gamma,\gamma,\gamma,\delta)$ for some scalars $\alpha,\beta,\gamma,\delta\in\real$.  Since $\rank(P_\pi)=k=4$, there are $4$ linearly independent eigenvectors of $L$ of the form of $x$ above.  One such vector is of course the all ones vector $e\in\real^{11}$ in which case $\alpha=\beta=\gamma=\delta=1$.  In general, since an eigenvector $x$ of the form above (with eigenvalue $\lambda\neq 0$) is orthogonal to $e$ we must have $4\alpha+2\beta+4\gamma+\delta=0$.  On the other hand, $\ker(P^T)$ is a 7-dimensional subspace of $\real^{11}$ and it can be verified that $\ker(P^T)$ has a basis consisting of eigenvectors of $L$.  Since $C_4=\{v_{11}\}$ is a singleton cell, all eigenvectors of $L$ in $\ker(P^T)$ have a zero in the last entry.  It follows that if, for example, $b=(0,0,0,0,1,1,0,0,0,0,1)\in \img(P)$ then $b$ is orthogonal to every eigenvector of $L$ in $\ker(P^T)$ and thus $(L,b)$ is uncontrollable by the PBH test.
\end{example}

There is a well-known relationship between the eigenvalues/eigenvectors of $L$ and $L_\pi$, namely that $(\ez,\lambda)$ is an eigenvector/eigenvalue pair of $L_\pi$ if and only if $(P\ez,\lambda)$ is an eigenvector/eigenvalue pair of $L$, where $P$ is the characteristic matrix of $\pi$ \cite{DC-CD-PR:07}.  Now, it is possible that the quotient graph $G_\pi$ itself contains an AEP and in this case an AEP of $G_\pi$ induces an AEP of $G$ in the following way.  Let $\pi=\{C_1,\ldots,C_k\}$ be a partition of $V$ and let $\rho=\{S_1,\ldots,S_{m}\}$ be a partition of $\pi$.  We define the \textit{$\rho$-merge of $\pi$} as the partition $\pi_\rho=\{\overline{C}_1,\overline{C}_2,\ldots,\overline{C}_{m}\}$ of $V$ where, for $j=1,2,\ldots,{m}$, we have $$\overline{C}_j = \bigcup_{C_i\in S_j} C_i.$$ 
Roughly speaking, the $\rho$-merge of $\pi$ is simply obtained by ``flattening out'' each cell $S_j$ of $\rho$.  For example, if $\pi=\{C_1,C_2,C_3,C_4\}=\{\{1,2,3\},\{4,5\},\{6,7\},\{8\}\}$ and $\rho=\{S_1,S_2\}=\{\{C_1,C_4\},\{C_2,C_3\}\}$ then the $\rho$-merge of $\pi$ is $\pi_{\rho} = \{\{1,2,3,8\},\{4,5,6,7\},\}$.  Notice that the partition $\pi$ is \textit{finer} than $\pi_\rho$, i.e., every cell in $\pi$ is a subset of a cell of $\pi_\rho$.   Now, if $\pi$ is an AEP of $G$ and $\rho$ is an AEP of $G_\pi$, then it is known that the $\rho$-merge of $\pi$ is an AEP of $G$ \cite[Prop.\ 1]{CA-BG:17}.

If $\rho$ is an AEP of $G_{\pi}$ then, by Theorem~\ref{thm:equi-lap}, $\img(P_\rho)$ is $L_\pi$-invariant but it does not generally hold that the orthogonal complement $\img(P_{\rho})^\perp$ is $L_{\pi}$-invariant since $L_{\pi}$ is not generally a symmetric matrix.  Hence, the discussion in Remark~\ref{rmk:eig-struct} regarding the eigenvector structure of $L$ does not generally hold for $L_\pi$, i.e., the eigenvectors of $L_\pi$ do not generally split into those that are constant on the cells of $\rho$ and those that sum to zero on the cells of $\rho$.  There is, however, a case in which the eigenvectors of $L_\pi$ do split into those in $\img(P_\rho)$ and $\img(P_{\rho})^\perp$, and this case is present in the network topologies that we characterize.  With this in mind we introduce the following notion.
\begin{definition}
Let $\pi=\{C_1,C_2,\ldots,C_k\}$ be a partition of $V$.  A partition $\rho=\{S_1,S_2,\ldots,S_{m}\}$ of $\pi$ is said to be \textit{$\pi$-regular} if each cell $S_j\in \rho$ consists of cells of $\pi$ of the same cardinality.
\end{definition}
In other words, if $S_j=\{C_{j,1},C_{j,2},\ldots,C_{j,m_j}\}$ then $\rho$ is \textit{$\pi$-regular} if $|C_{j,1}|=|C_{j,2}|=\cdots=|C_{j,m_j}|$, and this holds for all $j=1,2,\ldots,m$.  As an example, if $V=\{1,2,\ldots,12\}$ and $\pi=\{C_1,C_2,\ldots,C_6\}$ where $C_1=\{1,2,3\}$, $C_2=\{4,5\}$, $C_3=\{6\}$, and $C_4=\{7,8\}$, $C_5=\{9\}$, $C_6=\{10,11,12\}$ then $\rho=\{\{C_1,C_6\}, \{C_2,C_4\}, \{C_3,C_5\}\}$ is $\pi$-regular since $|C_1|=|C_6|$, $|C_2|=|C_4|$, and $|C_3|=|C_5|$.  The relationship between $\pi$-regularity of $\rho$ and the $L_\pi$-invariance of $\img(P_\rho)^\perp$ is then given in the following lemma whose proof can be found in Appendix~\ref{app:lemma-regular}.
\begin{lemma}\label{lem:pi-regular}
Let $G$ be a simple graph.  Let $\pi$ be an AEP of $G$ and let $\rho$ be an AEP of the quotient graph $G_\pi$.  If $\rho$ is $\pi$-regular then $\img(P_\rho)^\perp=\ker(P_\rho^T)$ is $L_\pi$-invariant.
\end{lemma}
The upshot of Lemma~\ref{lem:pi-regular} is that even though $L_\pi$ may not be a symmetric matrix, the discussion in Remark~\ref{rmk:eig-struct} regarding the splitting structure of the eigenvectors of $L$ is also applicable to $L_\pi$ provided $\rho$ is a $\pi$-regular AEP of $G_\pi$.  Hence, if the conditions of Lemma~\ref{lem:pi-regular} are satisfied, then the eigenvectors of $L_\pi$ can be partitioned into two classes, those contained in $\img(P_\rho)$ and the others contained in $\ker(P_\rho^T)$.

\section{Twin Graphs}\label{sec:twins}

One of the main structural properties possessed by the network topologies that we characterize is the existence of many \textit{twin vertices}.  The vertices $u,v\in V(G)$ are \textit{twins} if $N(u)\backslash \{v\} = N(v)\backslash\{u\}$.  We note that it is possible for $u$ to be twins with multiple vertices, that is, that $\{u,v\}$ and $\{u,w\}$ are twins with $v\neq w$.  In this case, either $\{u,v,w\}$ are all mutually adjacent or non-adjacent.  The existence of twin vertices induces an equitable partition as follows.  Recall that a permutation $\sigma:V\rightarrow V$ is an \textit{automorphism}, or \textit{symmetry}, of the graph $G=(V,E)$ if $\{u,v\}\in E$ if and only if $\{\sigma(u),\sigma(v)\}\in E$ for all $u,v\in V$.  We denote by $\aut(G)$ the group of automorphisms of $G$.  It is clear that if $u$ and $v$ are twin vertices then the permutation $\sigma:V\rightarrow V$ that transposes $u$ and $v$ and fixes all other vertices is an automorphism of $G$.  Moreover,  if without loss of generality $u=v_1$ and $v=v_2$, then the partition $\pi=\{\{v_1,v_2\}, \{v_3\}, \{v_4\}, \ldots, \{v_n\}\}$ is an equitable partition of $G$.  The following summarizes the relationship between twin vertices and eigenvectors/eigenvalues of $L$ (see \cite[pg. 46]{TB-JL-PS:07} for a proof).

\begin{proposition}[\cite{TB-JL-PS:07}]\label{prop:twin-eig}
Let $G=(V,E)$ be a simple graph with Laplacian matrix $L$, let $u,v\in V$, and let $\sigma:V\rightarrow V$ be the permutataion that transposes $u$ and $v$ and fixes all other vertices.  Then $\sigma$ is an automorphism of $G$ if and only if $\vv=\ev_u-\ev_v$ is an eigenvector of $\lap$.  In this case, $L$ has corresponding integer eigenvalue $\lambda=\deg(u) + A_{u,v}$.
\end{proposition}

An eigenvector of $\lap$ of the form $\vv=\ev_u-\ev_v$ is known as a Faria eigenvector \cite{TB-JL-PS:07}.  Clearly, a graph containing the Faria eigenvector $\vv=\ev_u-\ev_v$ will result in the uncontrollability of $(L,b)$ for all binary vectors $b$ whose entries are equal on $u$ and $v$ since then $\langle \vv, b\rangle = b_u - b_v = 0$.  Note that clearly $x=e_u-e_v\in \ker(P_\pi^T)$ where $\pi=\{\{u,v\}, \{v_3\}, \{v_4\}, \ldots, \{v_n\}\}$.

We now obtain an upper bound on the number of twin vertices in a graph whose Laplacian matrix has simple eigenvalues.  To that end, we first recall that the \textit{order} of a permutation $\sigma:V\rightarrow V$ is the smallest integer $k$ such that $\sigma^k:=\sigma\circ\sigma\circ\cdots\circ\sigma=\textup{id}$, where in the composition $\sigma$ appears $k$-times.  
\begin{lemma}\label{lem:max-twins}
Suppose that $G$ is a simple graph on $n$ vertices and assume that $L $ has simple spectrum.  The following hold:
\begin{compactenum}[(i)]
\item If $\{u,v\}$ are twin vertices then $\deg(u)=\deg(v)<  n - 1$.
\item If $n$ is even then the maximal number of twins is $t = \frac{n}{2} -1$, and this bound is sharp.
\item If $n$ is odd then the maximal number of twins is $t=\frac{n-1}{2}$, and this bound is sharp.
\end{compactenum}
\end{lemma}
\begin{proof}
To prove (i), suppose that $\deg(u)=\deg(v)=n-1$ for a twin pair $\{u,v\}\subset V$.  Then by Proposition~\ref{prop:twin-eig}, $\lambda=n$ is an eigenvalue of $\lap$ with Faria eigenvector $\vv=\ev_{u}-\ev_{v}$.  On the other hand, it is straightforward to verify that $\tilde{\vv} =  -e + n e_u $ is also an eigenvector of $\lap$ affording the eigenvalue $\lambda=n$.  Thus, $\lambda = n$ has algebraic multiplicity at least $2$ which contradicts the assumption that $L$ has simple spectrum.  Hence, we must have $\deg(u)=\deg(v)<n-1$.

To prove (ii), we first recall that if $G$ has an automorphism of order $k\geq 3$ then $L$ has a repeated eigenvalue \cite[pg. 45]{DC-PR-SS:97}.  Now, if $G$ has more than $\frac{n}{2}$ twin pairs of vertices then by the pigeon-hole principle, there are at least two sets of twin vertices of the form $\{u,v\}$ and $\{u,w\}$ with $v\neq w$.  Then the cyclic automorphism $\sigma(u)=v$, $\sigma(v)=w$, and $\sigma(w)=u$ (and all other vertices held fixed) clearly has order three and consequently $L$ does not have simple spectrum.  Suppose now then that $G$ has exactly $\frac{n}{2}$ twins and that $L$ has simple spectrum.  We claim that every eigenvalue induced by a twin is even.  To see this, if $\{u,v\}$ are twins then $\deg(u)=2q_u+A_{u,v}$ where $q_u$ is the number of twins $u$ (and hence $v$) is adjacent to.  By Proposition~\ref{prop:twin-eig},  the eigenvalue induced by $\{u,v\}$ is $\lambda = \deg(u) + A_{u,v} = 2(q_u+A_{u,v})$, which is even, and proves the claim.  Since $L$ has simple spectrum and all the eigenvalues induced by the $\frac{n}{2}$ twins are even, it follows that $n$ is an eigenvalue of $L$. Thus, there exists a twin pair each of which has degree $n-1$ which contradicts part (i).  Hence, this proves that no simple graph with Laplacian simple eigenvalues has $\frac{n}{2}$ or more twin vertices.  To prove that $\frac{n}{2}-1$ is a sharp bound for the number of twins, it may be verified that the graph on $n=8$ vertices in Figure~\ref{fig:sample-suc} has simple Laplacian spectrum and has $k=\frac{n}{2}-1=3$ twins given by $\{v_1,v_2\}$, $\{v_3,v_4\}$, and $\{v_5,v_6\}$.  

The proof of (iii) is similar to (ii) and is omitted.  In this case, the bound $t=\frac{n-1}{2}$ is attained by the graph on $n=11$ vertices in Figure~\ref{fig:sample-suc} which has simple Laplacian spectrum and has $k=\frac{n-1}{2}=5$ twins given by $\{v_1,v_2\}$, $\{v_3,v_4\}$, $\{v_5,v_6\}$, $\{v_7,v_8\}$, and $\{v_9,v_{10}\}$.  
\end{proof}


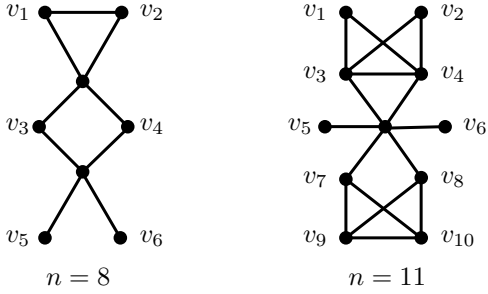
\begin{figure}
\centering
\begin{pspicture}(5,4)
\psset{xunit=1cm,yunit=1cm}

\rput[c](-0.5,2){
\psdots[dotsize=0.18](0.14,1.8037109)
\psdots[dotsize=0.18](1.16,1.8037109)
\psdots[dotsize=0.18](0.64,0.8837109)
\psdots[dotsize=0.18](1.14,-1.1962891)
\psdots[dotsize=0.18](0.14,-1.1962891)
\psdots[dotsize=0.18](0.64,-0.31628907)
\psdots[dotsize=0.18](1.24,0.28371093)
\psdots[dotsize=0.18](0.06,0.28371093)
\psline[linewidth=0.04cm](0.14,1.8037109)(1.16,1.8037109)
\psline[linewidth=0.04cm](0.14,1.783711)(0.66,0.86371094)
\psline[linewidth=0.04cm](1.16,1.783711)(0.64,0.8837109)
\psline[linewidth=0.04cm](0.04,0.28371093)(0.64,-0.31628907)
\psline[linewidth=0.04cm](1.24,0.26371095)(0.64,-0.31628907)
\psline[linewidth=0.04cm](0.66,0.86371094)(1.24,0.28371093)
\psline[linewidth=0.04cm](0.06,0.28371093)(0.66,0.8837109)
\psline[linewidth=0.04cm](0.64,-0.33628905)(0.14,-1.1962891)
\psline[linewidth=0.04cm](0.64,-0.33628905)(1.14,-1.1962891)
\rput(0.58145505,-1.711289){$n=8$}
\rput[r](-0.05,1.8037109){$v_1$}\rput[l](1.4,1.8037109){$v_2$}
\rput[r](-0.05,0.28371093){$v_3$}\rput[l](1.4,0.28371093){$v_4$}
\rput[r](-0.05,-1.1962891){$v_5$}\rput[l](1.4,-1.1962891){$v_6$}
}
\rput[c](-2.5,2){
\psdots[dotsize=0.18](6.14,1.8037109)
\psdots[dotsize=0.18](7.14,1.8037109)
\psdots[dotsize=0.18](7.14,0.98371094)
\psdots[dotsize=0.18](6.14,0.98371094)
\psdots[dotsize=0.18](6.66,0.28371093)
\psdots[dotsize=0.18](6.14,-1.1962891)
\psdots[dotsize=0.18](7.14,-1.1962891)
\psdots[dotsize=0.18](7.14,-0.39628905)
\psdots[dotsize=0.18](6.14,-0.41628906)
\psdots[dotsize=0.18](5.86,0.28371093)
\psdots[dotsize=0.18](7.46,0.28371093)
\rput[r](5.9,1.8037109){$v_1$}\rput[l](7.4,1.8037109){$v_2$}
\rput[r](5.9,0.98371094){$v_3$}\rput[l](7.4,0.98371094){$v_4$}
\rput[r](5.7,0.28371093){$v_5$}\rput[l](7.7,0.28371093){$v_6$}
\rput[r](5.9,-0.41628906){$v_7$}\rput[l](7.4,-0.41628906){$v_{8}$}
\rput[r](5.9,-1.1962891){$v_9$}\rput[l](7.4,-1.1962891){$v_{10}$}
\psline[linewidth=0.04cm](6.66,0.28371093)(5.86,0.28371093)
\psline[linewidth=0.04cm](6.68,0.26371095)(7.46,0.28371093)
\psline[linewidth=0.04cm](6.14,-1.1962891)(7.16,-1.1962891)
\psline[linewidth=0.04cm](6.14,1.8037109)(7.14,0.98371094)
\psline[linewidth=0.04cm](7.14,1.8037109)(7.14,0.98371094)
\psline[linewidth=0.04cm](6.14,0.98371094)(7.14,0.98371094)
\psline[linewidth=0.04cm](6.14,1.8037109)(6.14,0.98371094)
\psline[linewidth=0.04cm](7.14,1.8037109)(6.14,1.003711)
\psline[linewidth=0.04cm](7.14,0.98371094)(6.66,0.28371093)
\psline[linewidth=0.04cm](6.14,0.98371094)(6.66,0.28371093)
\psline[linewidth=0.04cm](6.66,0.26371095)(6.14,-0.41628906)
\psline[linewidth=0.04cm](6.66,0.28371093)(7.14,-0.39628905)
\psline[linewidth=0.04cm](6.14,-0.43628907)(6.14,-1.1962891)
\psline[linewidth=0.04cm](7.14,-0.41628906)(7.14,-1.1962891)
\psline[linewidth=0.04cm](7.14,-0.41628906)(6.14,-1.1962891)
\psline[linewidth=0.04cm](6.14,-0.41628906)(7.14,-1.1962891)
\rput(6.691455,-1.711289){$n=11$}
}
\end{pspicture} 
\caption{A sample of graphs with maximal twins and simple Laplacian eigenvalues}\label{fig:sample-suc}
\end{figure}

In view of Lemma~\ref{lem:max-twins}, we make the following definition.
\begin{definition}
A simple graph $G$ is said to be a \textit{twin graph} if $L=L(G)$ has simple spectrum and $G$ contains the maximal number of twins $t=\left\lfloor \frac{n-1}{2}\right\rfloor$.
\end{definition}
Let $w(n)$ be the number of twin graphs on $n$ vertices.  We have verified numerically that $w(n)=0$ for $n\leq 6$ and that $w(7) = 12$, $w(8) = 36$, $w(9) = 42$, and $w(10) = 924$.  We conjecture that twin graphs exists for all $n\geq 7$.  

If $G$ is a twin graph with vertex set $V=\{v_1,v_2,\ldots,v_n\}$, and $n$ is even, we assume that the $t=\frac{n}{2}-1$ twins are $C_1=\{v_1,v_2\}, C_2=\{v_3,v_4\},\ldots,C_t=\{v_{n-3},v_{n-2}\}$ and that $v_{n-1}$ and $v_n$ are the non-twin vertices.  In this case, we call $\ptw=\{C_1,C_2,\ldots,C_t,\{v_{n-1}\},\{v_n\}\}$ the \textit{twin partition} of $G$.  If $n$ is odd then we may assume that the $t=\frac{n-1}{2}$ twins are $C_1=\{v_1,v_2\},C_2=\{v_3,v_4\},\ldots, C_t=\{v_{n-2},v_{n-1}\}$, and $v_n$ will denote the non-twin vertex.  In this case, $\ptw=\{C_1,C_2,\ldots,C_t,\{v_n\}\}$ is the twin partition.  In either case, $\ptw$ is an equitable partition of $G$. 

We end this section with a discussion on the \textit{minimal controllability problem} for twin graphs \cite{AO:14}.  Given a linear time-invariant system $\dot{x} = Mx$ on $\real^n$, the minimal controllability problem is to find a smallest subset $\mathcal{I}=\{i_1,\ldots,i_p\}\subset \{1,2,\ldots,n\}$ such that if $B=\begin{bmatrix}e_{i_1}&\cdots&e_{i_p}\end{bmatrix}\in\real^{n\times p}$ then the linear time-invariant control system $\dot{x} = Mx + Bu$ is controllable.  It was proved in \cite{AO:14} that finding such a smallest $\mathcal{I}$ within a multiplicative factor of $c\log(n)$ is NP-hard for some absolute constant $c>0$, even when $M$ is symmetric.  It is known that if the maximum geometric multiplicity of the eigenvalues of $M$ is $q$ then $\rank(B) \geq q$ whenever $(M,B)$ is a controllable pair \cite[pg. 95]{ES:98}.  The existence of twin graphs shows that $q$ is in general a very poor lower bound for $\rank(B)=p$.  Indeed, if $G$ is a twin graph with twin cells $C_1,C_2,\ldots,C_t$ then any subset $\mathcal{I}$ chosen as input nodes such that $\mathcal{I}\cap C_i=\emptyset$ renders $(L,B)$ uncontrollable; this follows easily since the Faria eigenvector associated to $C_i$ is clearly orthogonal to every column of $B$.  As a consequence we obtain the following.
\begin{theorem}\label{thm:twin-mcp}
Let $G=(V,E)$ be a $n$-vertex twin graph with Laplacian matrix $L$ and let $t=\left \lfloor \frac{n-1}{2} \right \rfloor$ denote the number of twins in $G$.  Let $\mathcal{I}=\{i_1,\ldots,i_p\}\subset \{1,2,\ldots,n\}$ and let $B=\begin{bmatrix}e_{i_1}&\cdots&e_{i_p}\end{bmatrix}\in\real^{n\times p}$.  If $(L,B)$ is controllable then $p \geq t$.  
\end{theorem}
The punchline of Theorem~\ref{thm:twin-mcp} is that, at least for consensus-type dynamics, the generic case of simple eigenvalues in a graph matrix does not eliminate the necessity of controlling a significant fraction of the nodes to achieve network controllability.

\section{Strongly Uncontrollable Graphs}\label{sec:main}
In this section we present the main result of this paper.  For completeness, we formally state the definition of a strongly uncontrollable graph.

\begin{definition}
Let $G$ be a simple connected graph and suppose that $L=L(G)$ has simple spectrum.  We say that $G$ is \textit{strongly uncontrollable} if the pair $(L, b)$ is uncontrollable for every $b\in \{0,1\}^n$.
\end{definition}
Strong uncontrollability can also be defined using the adjacency or signless Laplacian matrix, etc.  However, the Laplacian matrix has the key property that it has the all ones vector $e$ as an eigenvector with corresponding eigenvalue $\lambda = 0$, and therefore $(L, b)$ is uncontrollable if and only if $(L, e-b)$ is uncontrollable (provided $n\geq 2$).  Also, as will be seen below, our sufficient condition for strong uncontrollability relies heavily on the orthogonality of eigenvectors of $L$ with the eigenvector $e$.  

Although a twin graph $G$ has many Faria eigenvectors, it is not necessarily a strongly uncontrollable graph.  In fact, we have verified that all twin graphs on $n=7$ vertices are not strongly uncontrollable and only 10 of the $w(8)=36$ twin graphs on $n=8$ vertices are strongly uncontrollable.  Below we give a sufficient condition for strong uncontrollability of twin graphs in terms of AEPs of the quotient graph $G_{\ptw}$.

\begin{theorem}\label{thm:suc-suff-even}
Let $G$ be a twin graph and suppose that $n=|V(G)|$ is even.  Let $\ptw=\{C_1,C_2,\ldots,C_t,\{v_{n-1}\},\{v_n\}\}$ denote the twin partition of $G$.  If the quotient graph $G_{\ptw}$ has an AEP of the form $\rho=\{S_1,S_2,\ldots,S_{{m}}, \{\{v_{n-1}\},\{v_n\}\}\}$ then $G$ is strongly uncontrollable.
\end{theorem}
\begin{proof}
By the PBH eigenvector controllability test, we must show that each binary vector $b\in\{0,1\}^n$ is orthogonal to some eigenvector of $L$.  If $b$ is a binary vector such that $b_{u} = b_{v}$ where $\{u,v\}=C_i$ for some $i\in\{1,2,\ldots,t\}$, then $b$ is orthogonal to the Faria eigenvector $e_u - e_v$, and thus a necessary condition for controllability is that $b_u \neq b_v$ for all twin pairs $\{u,v\}$.  To prove the theorem we will show the existence of two eigenvectors $\vv$ and $\tilde{\vv}$ of $L$ that are orthogonal to the binary vectors that have exactly one non-zero entry on each twin cell of $\ptw$.  To that end, let $P_{\ptw}$ be the characteristic matrix of the twin partition $\ptw$ and let $\xi_i$ be the characteristic vector of the cell $C_i$ for $i=1,2,\ldots,t$.  Then clearly
\[
\img(P_{\pi^*}) = \left\{\alpha\ev_{n-1}+\beta\ev_n + \sum_{i=1}^t \gamma_i \xi_i\;:\; \alpha,\beta,\gamma_i\in\real\right\}.
\]
Since $\{S_1,\ldots,S_{m}\}$ is a partition of the set $\{C_1,C_2,\ldots,C_t\}$, we can write  for each $j\in\{1,2,\ldots,{m}\}$ that $S_j = \{ C_{j,1}, C_{j,2},\ldots, C_{j,|S_j|}\}$,  where $C_{j,i}\in \{C_1,C_2,\ldots, C_t\}$ for all $i\in\{1,2,\ldots,|S_j|\}$.  We may therefore write an arbitrary vector in $\img(P_{\pi^*})$ in the form
\[
\alpha\ev_{n-1}+\beta\ev_n + \sum_{j=1}^{m} \sum_{i=1}^{|S_j|} \gamma_{j,i} \, \xi_{j,i}
\]  
where $\xi_{j,i}$ is the characteristic vector of the cell $C_{j,i}\in S_j$, and $\alpha,\beta,\gamma_{j,i}\in\real$.  

The partition $\rho=\{S_1,S_2,\ldots,S_{{m}},\{\{v_{n-1}\},\{v_n\}\}\}$ is $\ptw$-regular; indeed, each set $S_j$ contains cells of $\pi$ that have cardinality two (twin cells) and the set $\{\{v_{n-1}\},\{v_n\}\}$ clearly consists of cells of $\pi^*$ of the same cardinality.  Thus by Lemma~\ref{lem:pi-regular}, there exists an eigenvector $\ez$ of $\lap(G_{\ptw})$ such that $P_\rho^T y = 0$.  Therefore, the eigenvector $\vv=P_{\pi^*}\ez$ of $L $ can be written in the form
\[
\vv = \alpha(\ev_{n-1}-\ev_n) + \sum_{j=1}^{m} \sum_{i=1}^{|S_j|} \gamma_{j,i}\, \xi_{j,i}
\]
where for each $j\in\{1,2,\ldots,{m}\}$ we have 
\begin{equation}\label{eqn:alpha_ij}
\sum_{i=1}^{|S_j|}\gamma_{j,i}=0.
\end{equation}
We note that there exists at least one  $j\in\{1,2,\ldots,{m}\}$ such that $|S_j|\geq 2$ (i.e., at least one set $S_j$ is not a singleton cell); if not then $\gamma_{j,i}=0$ for all $i,j$ and therefore $\vv=\alpha(\ev_{n-1}-\ev_n)$ which implies that $\{v_{n-1},v_n\}$ is a twin cell of $G$ which contradicts the maximality of $t$.

Since $\rho$ is an AEP of $G_{\pi^*}$, there exists an eigenvector $\tilde{\ez}\neq\ones$ of $\lap(G_{\ptw})$ such that $\tilde{y}\in\img(P_\rho)$.  Hence, the eigenvector $\tilde{\vv}=P_{\pi^*}\tilde{\ez}$ of $L $ takes the form
\[
\tilde{\vv} = \beta(\ev_{n-1}+\ev_n) + \sum_{j=1}^{m} \mu_j \sum_{i=1}^{|S_j|} \xi_{j,i}
\]
for $\mu_j\in\real$.  Since $\tilde{\vv}$ is orthogonal to the all ones eigenvector of $L$, and since each vector $\xi_{j,i}$ is a sum of two distinct standard basis vectors, we have 
\[
0 = \langle\tilde{\vv},\ones\rangle = 2\beta+\sum_{j=1}^{{m}} 2 |S_j| \mu_j
\]
which simplifies to
\begin{equation}\label{eqn:beta}
\beta+\sum_{j=1}^{{m}}  |S_j| \mu_j=0.
\end{equation}

Now let $\bv\in\{0,1\}^n$.  As already mentioned at the beginning of the proof, we need only consider the case that $b$ has exactly one non-zero entry on each twin cell.  Hence, $\langle\bv,\xi_{j,i}\rangle = 1$ for all $i\in\{1,\ldots,|S_j|\}$ and all $j\in\{1,\ldots,{m}\}$.  There are three cases to consider.  If $b_{n-1}=b_n=1$ then from \eqref{eqn:alpha_ij} we have
\[
\langle \vv,\bv\rangle = \alpha-\alpha + \sum_{j=1}^{m} \sum_{i=1}^{|S_j|} \gamma_{j,i} =0.
\]
If $b_{n-1}=b_n=0$ then again from \eqref{eqn:alpha_ij} we have
\[
\langle \vv,\bv\rangle =  \sum_{j=1}^{m} \sum_{i=1}^{|S_j|} \gamma_{j,i} =0.
\]
Lastly, if $b_{n-1}\neq b_n$  then from \eqref{eqn:beta} we have
\[
\langle \tilde{\vv},\bv\rangle = \beta + \sum_{j=1}^{m} \mu_j |S_j| = 0.
\]
Thus, in any case, $\bv$ is orthogonal to an eigenvector of $L $ and thus $(L ,\bv)$ is uncontrollable.  Since $G$ has simple eigenvalues, by definition $G$ is a strongly uncontrollable graph.
\end{proof}

The case of an odd number of vertices is similar.
\begin{theorem}\label{thm:suc-suff-odd}
Let $G$ be a twin graph and suppose that $n=|V(G)|$ is odd.  Let $\ptw=\{C_1,C_2,\ldots,C_t,\{v_n\}\}$ denote the twin partition of $G$.  If $G_{\ptw}$ has a non-trivial AEP of the form $\rho=\{S_1,S_2,\ldots,S_{{m}}, \{ \{v_n\}\}\}$ then $G$ is strongly uncontrollable.  
\end{theorem}
%
We illustrate the previous results with an example.
\begin{example}
The twin graph $G$ on $n=11$ vertices shown in Figure~\ref{fig:quotient-suc} satisfies the assumptions of Theorem~\ref{thm:suc-suff-odd}; we also display the quotient graph $G_{\ptw}$ where $\ptw=V(G_{\ptw}) = \{C_1,C_2,C_3,C_4,C_5,\{v_{11}\}\}$ is the twin partition of $G$.   The Faria eigenvector associated to the twin $C_i$ is $x_i=e_{2i-1}-e_{2i}$ for $i=1,2,\ldots,5$.    The adjacency matrix of $G_{\ptw}$ is
\[
A_{\ptw} = \begin{bmatrix}0&2&0&0&0&0\\2&0&0&0&0&1\\0&0&0&0&0&1\\0&0&0&0&2&1\\0&0&0&2&0&0\\0&2&2&2&0&0\end{bmatrix}
\]
One can verify by inspection that $\rho=\{\{C_1,C_5\},\{C_2,C_4\}, \{C_3\},\{v_{11}\}\}$ is an AEP of $G_{\ptw}$.  The $\rho$-merge of $\ptw$ is $\ptw_\rho=\{\{1,2,9,10\},\{3,4,7,8\},\{5,6\},\{11\}\}$ and the quotient graph of $G_{\ptw_\rho}$ has adjacency matrix
\[
A(G_{\ptw_\rho}) = \begin{bmatrix}0&2&0&0\\2&0&0&1\\0&0&0&1\\0&4&2&0\end{bmatrix}.
\]

\begin{figure}
\centering
\begin{pspicture}(8,6)
\psset{xunit=1.5cm,yunit=1.5cm}
\rput[c](-5.5,2){
\psdots[dotsize=0.18](6.14,1.8037109)\rput[r](6,1.8){$v_1$}
\psdots[dotsize=0.18](7.14,1.8037109)\rput[l](7.3,1.8){$v_2$}
\psdots[dotsize=0.18](7.14,0.98371094)\rput[r](6,0.98){$v_3$}
\psdots[dotsize=0.18](6.14,0.98371094)\rput[l](7.3,0.98){$v_4$}
\psdots[dotsize=0.18](6.66,0.28371093)\rput[r](5.7,0.28){$v_5$}
\psdots[dotsize=0.18](6.14,-1.1962891)\rput[l](7.6,0.28){$v_6$}
\psdots[dotsize=0.18](7.14,-1.1962891)\rput[r](6,-0.4){$v_7$}
\psdots[dotsize=0.18](7.14,-0.39628905)\rput[l](7.3,-0.4){$v_8$}
\psdots[dotsize=0.18](6.14,-0.41628906)\rput[r](6,-1.2){$v_9$}
\psdots[dotsize=0.18](5.86,0.28371093)\rput[l](7.3,-1.2){$v_{10}$}
\psdots[dotsize=0.18](7.46,0.28371093)\rput[c](6.7,-0.1){\small $v_{11}$}
\psline[linewidth=0.04cm](6.66,0.28371093)(5.86,0.28371093)
\psline[linewidth=0.04cm](6.68,0.26371095)(7.46,0.28371093)
\psline[linewidth=0.04cm](6.14,-1.1962891)(7.16,-1.1962891)
\psline[linewidth=0.04cm](6.14,1.8037109)(7.14,0.98371094)
\psline[linewidth=0.04cm](7.14,1.8037109)(7.14,0.98371094)
\psline[linewidth=0.04cm](6.14,0.98371094)(7.14,0.98371094)
\psline[linewidth=0.04cm](6.14,1.8037109)(6.14,0.98371094)
\psline[linewidth=0.04cm](7.14,1.8037109)(6.14,1.003711)
\psline[linewidth=0.04cm](7.14,0.98371094)(6.66,0.28371093)
\psline[linewidth=0.04cm](6.14,0.98371094)(6.66,0.28371093)
\psline[linewidth=0.04cm](6.66,0.26371095)(6.14,-0.41628906)
\psline[linewidth=0.04cm](6.66,0.28371093)(7.14,-0.39628905)
\psline[linewidth=0.04cm](6.14,-0.43628907)(6.14,-1.1962891)
\psline[linewidth=0.04cm](7.14,-0.41628906)(7.14,-1.1962891)
\psline[linewidth=0.04cm](7.14,-0.41628906)(6.14,-1.1962891)
\psline[linewidth=0.04cm](6.14,-0.41628906)(7.14,-1.1962891)
\rput(6.691455,-1.711289){$G$}
}
\rput[c](-2,2){
\psdots[dotsize=0.18](5.86,1.8037109)\rput[r](5.7,1.8){\small $C_1$}
\psdots[dotsize=0.18](5.86,0.98371094)\rput[r](5.7,0.98){\small $C_2$}
\psdots[dotsize=0.18](5.86,0.28371093)\rput[r](5.7,0.28){\small $\{v_{11}\}$}
\psdots[dotsize=0.18](6.66,0.28371093)\rput[l](6.8,0.28){\small $C_3$}
\psdots[dotsize=0.18](5.86,-0.41628906)\rput[r](5.7,-0.4){\small $C_4$}
\psdots[dotsize=0.18](5.86,-1.1962891)\rput[r](5.7,-1.2){\small $C_5$}
\psline[linewidth=0.04cm](5.86,-1.1962891)(5.86,1.8037109)
\psline[linewidth=0.04cm](5.86,0.28371093)(6.66,0.28371093)
\rput(5.86,-1.711289){$G_{\ptw}$}
\psline[linewidth=1pt,arrowsize=1pt 3]{->}(5.86,1.8037109)(5.86,1.1)\rput[l](5.95,1.2){\footnotesize $2$}
\psline[linewidth=1pt,arrowsize=1pt 3]{->}(5.86,0.983)(5.86,1.7)\rput[l](5.95,1.65){\footnotesize $2$}
\psline[linewidth=1pt,arrowsize=1pt 3]{->}(5.86,0.983)(5.86,0.4)\rput[l](5.95,0.8){\footnotesize $2$}
\psline[linewidth=1pt,arrowsize=1pt 3]{->}(5.86,0.283)(5.86,0.9)\rput[r](5.8,0.5){\footnotesize $1$}
\psline[linewidth=1pt,arrowsize=1pt 3]{->}(5.86,0.283)(6.55,0.283)\rput[r](6.55,0.42){\footnotesize $2$}
\psline[linewidth=1pt,arrowsize=1pt 3]{->}(6.66,0.283)(6,0.283)\rput[r](6.1,0.4){\footnotesize $1$}
\psline[linewidth=1pt,arrowsize=1pt 3]{->}(5.86,0.283)(5.86,-0.32)\rput[l](5.95,0.1){\footnotesize $1$}
\psline[linewidth=1pt,arrowsize=1pt 3]{->}(5.86,-0.416)(5.86,0.15)\rput[l](5.95,-0.23){\footnotesize $2$}
\psline[linewidth=1pt,arrowsize=1pt 3]{->}(5.86,-0.4)(5.86,-1.1)\rput[l](5.95,-0.6){\footnotesize $2$}
\psline[linewidth=1pt,arrowsize=1pt 3]{->}(5.86,-1.2)(5.86,-0.5)\rput[l](5.95,-1){\footnotesize $2$}
}
\end{pspicture} 
\caption{$G$ and $G_{\ptw}$ satisfying Theorem~\ref{thm:suc-suff-odd}}\label{fig:quotient-suc}
\end{figure}
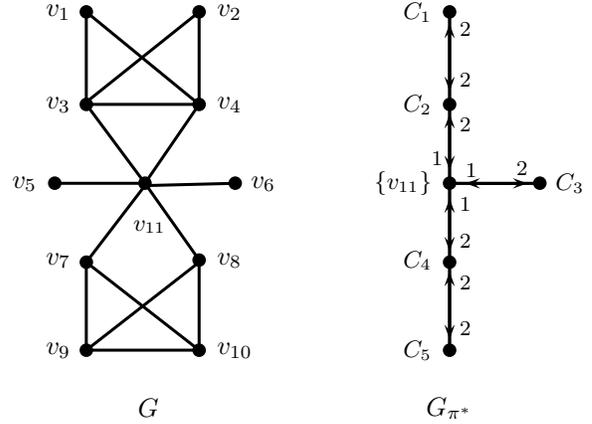
\end{example}

Theorems~\ref{thm:suc-suff-even} and \ref{thm:suc-suff-odd} are not necessary for strong uncontrollability as will be seen in the next section where we consider the stability of strong uncontrollability to vertex additions.  Through numerical investigations, however, we have found that all strongly uncontrollable graphs up to $n=12$ vertices have at least three twin pairs.  For the adjacency matrix, twin vertices induce the eigenvalue $\lambda = 0$ or $\lambda = -1$ depending on whether the twin vertices are adjacent or not adjacent.  Thus, any graph with three twin vertices has an adjacency matrix with at least one repeated eigenvalue.  This observation leads us to conjecture that strongly uncontrollable graphs using the adjacency matrix $A$ do not exist.  As for the signless Laplacian matrix $Q=D+A$, we have not found any strongly uncontrollable graphs up to $n=10$ vertices.  Although our numerical investigations are for very small $n$, they suggest that strong uncontrollability is a property that may only exist in consensus-type network dynamics. 

\section{Uncontrollability under perturbations}\label{sec:perturbations}
In this section, we analyze the stability of the strong uncontrollability property to vertex additions.  Our results rely on the following well-known result concerning the Laplacian eigenvalues under the graph join operation.  Given two simple graphs $G_1=(V_1,E_1)$ and $G_2=(V_2,E_2)$ such that $V_1\cap V_2=\emptyset$, the \textit{join} of $G_1$ and $G_2$ is the graph $G=G_1\vee G_2$ with vertex set $V(G)=V_1\cup V_2$ and edge set $E(G)=E_1\cup E_2 \cup \{\{u,v\}\;|\; u\in V_1,\; v\in V_2\}$.
\begin{theorem}[\cite{TB-JL-PS:07}]\label{thm:join}
Let $G_1$ and let $G_2$ be simple graphs on disjoint sets of $n_1$ and $n_2$ vertices, respectively.  Let $\lap_1$ be the Laplacian matrix of $G_1$, with eigenvectors $\vv_1,\vv_2,\ldots,\vv_{n_1}$ and corresponding eigenvalues $0=\alpha_1\leq\alpha_2\leq\cdots\leq\alpha_{n_1}$.  Let $\lap_2$ be the Laplacian matrix of $G_2$, with eigenvectors $\ez_1,\ez_2,\ldots,\ez_{n_2}$ and corresponding eigenvalues $0=\beta_1\leq\beta_2\leq\cdots\leq\beta_{n_2}$.  Let $\lap$ be the Laplacian matrix of the join graph $G=G_1\vee G_2$.  The following hold:
\begin{compactenum}[(i)]
\item For all $i=2,\ldots,n_1$, $\begin{bmatrix}\vv_i^T & \zeros_{n_2}^T\end{bmatrix}^T$ is an eigenvector of $\lap$ with eigenvalue $n_2+\alpha_i$.
\item For all $j=2,\ldots,n_2$, $\begin{bmatrix}\zeros_{n_1}^T & \ez_j^T\end{bmatrix}^T$ is an eigenvector of $\lap$ with eigenvalue $n_1+\beta_j$.  
\item $\begin{bmatrix}n_2\vv_1^T & -n_1\ez_1^T\end{bmatrix}^T$ is an eigenvector of $\lap$ with eigenvalue $n_1+n_2$.
\end{compactenum}
\end{theorem}

The following theorem describes how a strongly uncontrollable graph can be constructed from a lower order strongly uncontrollable graph while preserving its automorphism group.  Henceforth, we denote by $S_n$ the symmetric group on $\{1,2,\ldots,n\}$, that is, the group of all permutations on $\{1,2,\ldots,n\}$.
\begin{theorem}\label{thm:suc-new}
Let $G$ be a strongly uncontrollable graph on $n$ vertices.  Let $\widetilde{G}$ be the graph on $n+1$ vertices obtained from $G$ by adding a vertex and connecting it to all the vertices of $G$.  Then $\widetilde{G}$ is a strongly uncontrollable graph if and only if the spectral radius of $L $ is less than $n$.  In this case, $\aut(G)$ and $\aut(\widetilde{G})$ are equal when viewed as subgroups of the symmetric group $S_{n+1}$.
\end{theorem}
\begin{proof}
%
Let $\vv_1,\vv_2,\ldots,\vv_n$ be the eigenvectors of $\lap=L(G) $ with corresponding eigenvalues $0=\lambda_1<\lambda_2<\cdots<\lambda_n$.  Applying Theorem~\ref{thm:join} to $G$ and the graph with one vertex, the set $\{e,\begin{bmatrix}\vv_2^T & 0\end{bmatrix}^T,\ldots,\begin{bmatrix}\vv_n^T & 0\end{bmatrix}^T,\begin{bmatrix}e^T & -n\end{bmatrix}^T\}$ consists of mutually orthogonal eigenvectors of $\widetilde{\lap}=\lap(\widetilde{G})$ with eigenvalues $0=\lambda_1<\lambda_2+1<\cdots<\lambda_n+1\leq n+1$.  If $\lambda_n<n$ then $n+1$ is a simple eigenvalue of $\widetilde{\lap}$.  Consequently, the eigenvalues of $\widetilde{\lap}$ are distinct.

Now, if $\bv\in\{0,1\}^{n+1}$ then since $G$ is strongly uncontrollable there exists an eigenvector $\vv_i\neq \vv_1$ of $L$ such that $\begin{bmatrix}\vv_i^T & 0\end{bmatrix}^T$ is orthogonal to $b$.  This proves that $(\widetilde{\lap},\bv)$ is uncontrollable for every $\bv\in\{0,1\}^{n+1}$.

To prove the second claim, since the spectral radius of $\lap$ is less than $n$, no vertex of $G$ has degree $n-1$.  Hence, the vertex $v_{n+1}$ that was added to $G$ to form $\widetilde{G}$ is the only vertex of $\widetilde{G}$ with degree $n$.  It follows that any automorphism of $\widetilde{G}$ must fix $v_{n+1}$, and this proves that $\aut(G)$ and $\aut(\widetilde{G})$ are equal when viewed as subgroups of $S_{n+1}$.

To prove the converse statement, assume that $\widetilde{G}$ is strongly uncontrollable.  Then by definition $\wlap$ has simple spectrum.  Then since $n+1$ is an eigenvalue of $\wlap$ it follows that $\lambda_n<n$.  
\end{proof}
Theorem~\ref{thm:suc-new} shows that if $G$ is a twin graph on an even number of vertices then $\widetilde{G}$ is not a twin graph, and thus showing that the property of being a twin graph is not necessary for strong uncontrollability.

\begin{example}
Let $s(n)$ be the number of strongly uncontrollable graphs on $n$ vertices.   We have numerically verified that $s(8) = 10$ and $s(9)=12$.  All 10 strongly uncontrollable graphs for $n=8$ have spectral radius less than $n$ and are all twin graphs. Hence, each strongly uncontrollable graph on $n=8$ vertices induces a strongly uncontrollable graph on $n=9$ vertices via Theorem~\ref{thm:suc-new}, none of which is a twin graph.  One such pair is displayed in Figure~\ref{fig:suc-pair}.  The other 2 uncontrollable graphs on $n=9$ vertices are twin graphs.

\begin{figure}[h]
\centering
\begin{pspicture}(5,4)
\psset{xunit=1cm,yunit=1cm}
\rput[c](0,2){
\psdots[dotsize=0.18](0.14,1.8037109)
\psdots[dotsize=0.18](1.16,1.8037109)
\psline[linewidth=0.04cm](0.14,1.8037109)(1.16,1.8037109)
\psdots[dotsize=0.18](0.64,0.8837109)
\psdots[dotsize=0.18](1.14,-1.1962891)
\psdots[dotsize=0.18](0.14,-1.1962891)
\psdots[dotsize=0.18](0.64,-0.31628907)
\psdots[dotsize=0.18](1.24,0.28371093)
\psdots[dotsize=0.18](0.06,0.28371093)
\psline[linewidth=0.04cm](0.14,1.783711)(0.66,0.86371094)
\psline[linewidth=0.04cm](1.16,1.783711)(0.64,0.8837109)
\psline[linewidth=0.04cm](0.04,0.28371093)(0.64,-0.31628907)
\psline[linewidth=0.04cm](1.24,0.26371095)(0.64,-0.31628907)
\psline[linewidth=0.04cm](0.66,0.86371094)(1.24,0.28371093)
\psline[linewidth=0.04cm](0.06,0.28371093)(0.66,0.8837109)
\psline[linewidth=0.04cm](0.64,-0.33628905)(0.14,-1.1962891)
\psline[linewidth=0.04cm](0.64,-0.33628905)(1.14,-1.1962891)
\rput(0.58145505,-1.711289){$n=8$}
}
\rput[c](2,2){
\psdots[dotsize=0.18](2.14,1.8037109)
\psdots[dotsize=0.18](3.16,1.8037109)
\psline[linewidth=0.04cm](2.14,1.8037109)(3.16,1.8037109)
\psdots[dotsize=0.18](2.64,0.8837109)
\psdots[dotsize=0.18](3.14,-1.1962891)
\psdots[dotsize=0.18](2.14,-1.1962891)
\psdots[dotsize=0.18](2.64,-0.31628907)
\psdots[dotsize=0.18](3.24,0.28371093)
\psdots[dotsize=0.18](2.06,0.28371093)
\psline[linewidth=0.04cm](2.14,1.783711)(2.66,0.86371094)
\psline[linewidth=0.04cm](3.16,1.783711)(2.64,0.8837109)
\psline[linewidth=0.04cm](2.04,0.28371093)(2.64,-0.31628907)
\psline[linewidth=0.04cm](3.24,0.26371095)(2.64,-0.31628907)
\psline[linewidth=0.04cm](2.66,0.86371094)(3.24,0.28371093)
\psline[linewidth=0.04cm](2.06,0.28371093)(2.66,0.8837109)
\psline[linewidth=0.04cm](2.64,-0.33628905)(2.14,-1.1962891)
\psline[linewidth=0.04cm](2.64,-0.33628905)(3.14,-1.1962891)
\psdots[dotsize=0.18](2.66,0.28371093)
\psline[linewidth=0.04cm](2.66,0.86371094)(2.66,0.28371093)
\psline[linewidth=0.04cm](2.66,-0.31628907)(2.66,0.26371095)
\psline[linewidth=0.04cm](2.06,0.28371093)(3.26,0.28371093)
\psline[linewidth=0.04cm](3.16,1.8037109)(2.68,0.28371093)
\psline[linewidth=0.04cm](2.14,1.783711)(2.66,0.28371093)
\psline[linewidth=0.04cm](2.66,0.26371095)(2.14,-1.1962891)
\psline[linewidth=0.04cm](2.66,0.28371093)(3.14,-1.1962891)
\rput(2.6414552,-1.731289){$n=9$}
}
\end{pspicture} 
\caption{A strongly uncontrollable graph on $n=8$ vertices and its induced strongly uncontrollable graph on $n=9$ vertices using Theorem~\ref{thm:suc-new}.}
\label{fig:suc-pair}
\end{figure}
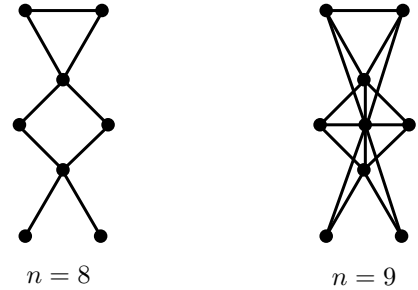
\end{example}
We now consider the case of adding two vertices.  If $\tau:\{1,2,\ldots,n\}\rightarrow\{1,2,\ldots,n\}$ is a permutation such that $\tau(i) = j$ and $\tau(j)=i$ (with $i\neq j$) and fixes all other integers (i.e., $\tau$ is a transposition), we denote $\tau$ by $\tau=(i\;j)$.
\begin{theorem}\label{thm:suc-new2}
Let $G$ be a strongly uncontrollable graph with vertex set $V =\{v_1,v_2,\ldots,v_n\}$.  Let $G_2=(\{v_{n+1},v_{n+2}\},\emptyset)$ be the empty graph on two vertices and let $\widetilde{G}=G\vee G_2$.  Then $\widetilde{G}$ is a strongly uncontrollable graph if and only if $n-2$ and $n$ are not eigenvalues of $L $.  In this case, $\aut(\widetilde{G})$ is generated by the union of a generating set of $\aut(G)$ and the transposition $\tau=(n+1\;\; n+2)$.
\end{theorem}
\begin{proof}
Let $\vv_1,\vv_2,\ldots,\vv_n$ be the eigenvectors of $\lap=L(G) $ with corresponding eigenvalues $0=\lambda_1<\lambda_2<\cdots<\lambda_n$.    The Laplacian matrix of $\widetilde{G}$ can be written as
\[
\wlap = \begin{bmatrix}(\lap+2I) & -\ones & -\ones \\[2ex] -\ones^T & n & 0\\[2ex] -\ones^T & 0 & n\end{bmatrix}.
\]
The vectors $\ones, \begin{bmatrix}\vv_2^T & 0\end{bmatrix}^T,\ldots, \begin{bmatrix}\vv_n^T & 0\end{bmatrix}^T$ are eigenvectors of $\wlap$ with eigenvalues $0=\lambda_1<\lambda_2+2<\cdots<\lambda_n+2$, and $\begin{bmatrix}2\ones^T & -n&-n\end{bmatrix}^T$ is an eigenvector of $\wlap$ with eigenvalue $n+2$.  The Faria eigenvector $\begin{bmatrix}\zeros^T & 1 & -1\end{bmatrix}^T$ of $\wlap$ has corresponding eigenvalue $n$.  Thus, if $n-2$ and $n$ are not eigenvalues of $\lap$ then $\wlap$ has simple spectrum.  Since $G$ is a strongly uncontrollable graph, every binary vector in $\{0,1\}^{n+2}$ is orthogonal to some eigenvector $\begin{bmatrix}\vv_i^T & 0 & 0\end{bmatrix}^T$ of $\wlap$.  Thus, $\wgph$ is also a strongly uncontrollable graph.  As in Theorem~\ref{thm:suc-new}, the converse statement is straightforward.

To prove the second statement, it is clear that every automorphism of $G$ can be extended to an automorphism of $\wgph$ by asking that it fix the vertices $v_{n+1}$ and $v_{n+2}$.  The transposition $\tau=(n+1\;\; n+2)$ is an automorphism of $\wgph$ and there are no other automorphisms of $\wgph$ that do not fix $v_{n+1}$ and $v_{n+2}$.  
\end{proof}

\begin{remark}
The procedure in Theorems~\ref{thm:suc-new} and \ref{thm:suc-new2} of taking a strongly uncontrollable graph $G$ and creating a new strongly uncontrollable graph by joining it to an empty graph with $n_2$ vertices cannot be extended to the case $n_2>2$.  Indeed, if $n_2>2$ vertices are joined to $G$ then $\wgph$ will contain an automorphism of order greater than two, and thus $\wgph$ would not have distinct eigenvalues.
\end{remark}

The following theorem can be seen as a complement of Theorem~\ref{thm:suc-new}.
\begin{theorem}\label{thm:one}
Let $G$ be a strongly uncontrollable graph with vertex set $V =\{v_1,v_2,\ldots,v_n\}$ and suppose that $\deg(v_n)=n-1$.  Let $\widetilde{G}$ be the graph on $n+1$ vertices obtained from $G$ by adding a vertex and connecting it only to $v_n$.  Then $\widetilde{G}$ is a strongly uncontrollable graph if and only if $1$ is not an eigenvalue of $L $.  Moreover, $\aut(G)$ and $\aut(\widetilde{G})$ are equal when viewed as subgroups of $S_{n+1}$.
\end{theorem}
\begin{proof}
Let $\vv_1,\vv_2,\ldots,\vv_n$ be eigenvectors of $\lap=L(G) $ with corresponding eigenvalues $0=\lambda_1<\lambda_2<\cdots<\lambda_n=n$.  Since $\deg(v_n)=n-1$ we may take $\vv_n=-e + ne_n.$ 
By orthogonality of eigenvectors of $L$, we have $0=\langle\vv_j,\vv_n\rangle=n\langle \vv_j, e_n\rangle$ for $2\leq j\leq n-1$, that is, $\dotprod{x_j, e_n}=e_n^Tx_j=0$.  Now, $\wlap=\lap(\widetilde{G})$ takes the form
\[
\wlap=\begin{bmatrix}\lap + \ev_n\ev_n^T & -\ev_n \\[2ex] -\ev_n^T & 1\end{bmatrix}
\]
and therefore if we set $\tilde{\vv}_j=\begin{bmatrix}\vv_j^T & 0\end{bmatrix}^T$, for $2\leq j\leq n-1$, we have
\[
\wlap\tilde{\vv}_j = \begin{bmatrix} \lap\vv_j + \ev_n\ev_n^T\vv_j\\[2ex] -\ev_n^T\vv_j\end{bmatrix} = \begin{bmatrix} \lap\vv_j\\[2ex] 0\end{bmatrix} = \lambda_j\tilde{\vv}_j.
\]
Hence, $\tilde{\vv}_j$ is an eigenvector of $\wlap$ with eigenvalue $\lambda_j$, for $2\leq j\leq n-1$.  Now, $\tilde{\vv}_{n+1}=\wlap\ev_{n}=\begin{bmatrix}-1&-1&\cdots -1&n&-1\end{bmatrix}^T\in\real^{n+1}$ is an eigenvector of $\wlap$ with eigenvalue $\lambda_{n+1}=n+1$.  Finally, consider the vector $\tilde{\vv}_n = \begin{bmatrix}-1&-1&\cdots&-1&0&n-1\end{bmatrix}^T \in \real^{n+1}$.  A straightforward calculation shows that $\tilde{\vv}_n$ is an eigenvector of $\wlap$ with eigenvalue $1$.  Hence, if $\lambda_j\neq 1$ for $2\leq j\leq n-1$, then $\wlap$ has simple eigenvalues $\{0,1,\lambda_2,\lambda_3,\ldots,\lambda_{n-1},n+1\}$.  The rest of the proof is similar to that of Theorem~\ref{thm:suc-new} and is omitted.
\end{proof}

\section{Conclusion}

In this paper, we have characterized network topologies under which the Laplacian consensus dynamics are uncontrollable for any subset of the nodes chosen as control inputs and that emit a common control signal.  In these network topologies, the lack of controllability is not due to repeated eigenvalues of the Laplacian matrix, but instead is characterized by structural properties of the network, namely, the existence of a maximal number of twin nodes and certain almost equitable partitions. We provided a sufficient condition for a network to contain this strong uncontrollability property and described network perturbations that leave the uncontrollability property invariant.  We also related our work with the minimal controllability problem and showed how these network topologies require the control of essentially half of the nodes for any chance of controllability.


%

\appendices
\section{Proof of Lemma~\ref{lem:pi-regular}}\label{app:lemma-regular}
Before we give the proof of Lemma~\ref{lem:pi-regular} we need some preliminary results. Let $\pi=\{C_1,C_2,\ldots,C_k\}$ be a partition of $V=\{v_1,\ldots,v_n\}$ and let $K=\diag(|C_1|,|C_2|,\ldots,|C_k|)\in\real^{k\times k}$.  It is easy to see that $K=P^T_\pi P_\pi$.    Let $\rho=\{S_1,S_2,\ldots,S_{m}\}$ be a partition of $\pi$ and let $\xi_j\in\{0,1\}^k$ be the characteristic vector of $S_j$.  If $\rho$ is $\pi$-regular (i.e., all cells in $S_j$ have the same cardinality) then clearly 
\[
K\xi_j = |C_i| \xi_j
\]
for any (and hence all) $C_i \in S_j$.  It follows then that $KP_\rho = P_\rho\tilde{K}$ where $\tilde{K}=\diag(|C_{1,1}|,|C_{2,1}|,\ldots,|C_{m,1}|)$ where $C_{j,1}\in S_j$ for $j=1,2,\ldots,m$.  An identical argument shows that $K^{-1}P_\rho = P_\rho \tilde{K}^{-1}$.  We can now prove Lemma~\ref{lem:pi-regular}.

\begin{proof}[Proof of Lemma~\ref{lem:pi-regular}]
Let $\pi=\{C_1,C_2,\ldots,C_k\}$ and let $K = P_\pi^TP_\pi=\textup{diag}(|C_1|, |C_2|, \ldots, |C_k|)$ as above.  Then from Theorem~\ref{thm:equi-lap}, we have that $L_\pi = K^{-1} P_\pi^T L P_\pi$.  Also, from Theorem~\ref{thm:equi-lap} applied to the quotient graph $G_\pi$ and the partition $\rho$, we have that $L_\pi P_\rho = P_\rho L_{\pi_\rho}$ where $\pi_\rho$ is the $\rho$-merge of $\pi$.  Then
\begin{align*}
L_\pi^T P_\rho &= (P_\pi^T L^T P_\pi K^{-1}) P_\rho \\
&= (P_\pi^T L P_\pi)( K^{-1} P_\rho)\\
&= (K L_\pi) (K^{-1}P_\rho) \\
&= (K L_\pi)( P_\rho \tilde{K}^{-1} )\\
&= K P_\rho L_{\pi_\rho} \tilde{K}^{-1}\\
&= P_\rho (\tilde{K} L_{\pi_\rho} \tilde{K}^{-1}).
\end{align*}
In other words, $P_\rho^T L_\pi = (\tilde{K} L_{\pi_\rho} \tilde{K}^{-1})^T P_\rho^T$, and thus if $P_\rho^T x=0$ then clearly $P_\rho^TL_\pi x =0$, i.e., $\ker(P_\rho^T)$ is $L_\pi$-invariant.
\end{proof}


\section*{Acknowledgment}
The author wishes to thank the anonymous referees for their comments and suggestions that improved the presentation of the paper.  The author acknowledges the support of the National Science Foundation under Grant No. ECCS-1700578.

\ifCLASSOPTIONcaptionsoff
  \newpage
\fi



\bibliographystyle{IEEEtran}
\bibliography{master_bib}
\end{document}